\documentclass[12pt,english]{amsart}
\usepackage{charter}
\usepackage[T1]{fontenc}
\usepackage[latin9]{inputenc}
\usepackage{babel}
\usepackage{enumitem}
\usepackage{amstext}
\usepackage{amsthm}
\usepackage{amssymb}
\usepackage{geometry}
\usepackage{setspace}
\usepackage[pdfusetitle,
 bookmarks=true,bookmarksnumbered=false,bookmarksopen=false,
 breaklinks=false,pdfborder={0 0 1},backref=false,colorlinks=false]
 {hyperref}

\makeatletter
\numberwithin{equation}{section}
\AtBeginDocument{}
\AtBeginDocument{
	\theoremstyle{plain}
	\newtheorem{question}[thm]{Question}
}

\usepackage{eucal}
\let\mathcal=\CMcal
\usepackage{dsfont}
\usepackage{graphicx}
\usepackage{mathtools} 

\setlist[enumerate,1]{label={(\arabic*)},ref={(\arabic*)}}
\setlist[enumerate,2]{label={\alph*.},ref={\alph*}}

\makeatother

\theoremstyle{plain}
\newtheorem{thmx}{\protect\theoremname}[section]
\newtheorem{thm}{\protect\theoremname}[section]
\theoremstyle{definition}
\newtheorem{defn}[thm]{\protect\definitionname}
\theoremstyle{remark}
\newtheorem{rem}[thm]{\protect\remarkname}
\theoremstyle{definition}
\newtheorem{example}[thm]{\protect\examplename}
\theoremstyle{plain}
\newtheorem{lem}[thm]{\protect\lemmaname}
\providecommand{\definitionname}{Definition}
\providecommand{\examplename}{Example}
\providecommand{\lemmaname}{Lemma}
\providecommand{\remarkname}{Remark}
\providecommand{\theoremname}{Theorem}

\begin{document}
\global\long\def\e{\varepsilon}%
\global\long\def\N{\mathbb{N}}%
\global\long\def\Z{\mathbb{Z}}%
\global\long\def\Q{\mathbb{Q}}%
\global\long\def\R{\mathbb{R}}%
\global\long\def\C{\mathbb{C}}%
\global\long\def\H{\mathcal{H}}%
\global\long\def\K{\mathcal{K}}%
\global\long\def\l{\lambda}%

\global\long\def\linspan{\operatorname{span}}%
\global\long\def\clinspan{\operatorname{\overline{span}}}%
\global\long\def\supp{\operatorname{supp}}%
\global\long\def\op{\mathrm{op}}%

\global\long\def\Rant{\text{\scalebox{1.15}{\ensuremath{\mathtt{R}}}}}%
\global\long\def\Sant{\text{\scalebox{1.15}{\ensuremath{\mathtt{S}}}}}%
\global\long\def\QG{\mathbb{G}}%
\global\long\def\G{\QG}%
\global\long\def\Cz{\mathrm{C}_{0}}%
\global\long\def\CzU{\mathrm{C}^{\mathrm{u}}_{0}}%
\global\long\def\CStar{\mathrm{C}^{*}}%
\global\long\def\VN{\operatorname{VN}}%
\global\long\def\Mor{\operatorname{Mor}}%
\global\long\def\M{\mathrm{M}}%
\global\long\def\Aut{\operatorname{Aut}}%
\global\long\def\CStarR{\mathrm{C}^{*}_{\mathrm{r}}}%
\global\long\def\tensormin{\otimes_{\mathrm{min}}}%
\global\long\def\tensorn{\mathbin{\overline{\otimes}}}%
\global\long\def\id{\mathrm{id}}%
\global\long\def\i{\id}%
\global\long\def\Cb{\mathrm{C}_{b}}%
\global\long\def\Cc{\mathrm{C}_{c}}%
\global\long\def\Ad{\operatorname{Ad}}%
\global\long\def\ev{\mathrm{ev}}%
\global\long\def\a{\alpha}%
\global\long\def\one{\mathds{1}}%
\global\long\def\d{\,\mathrm{d}}%
\global\long\def\dd{\mathrm{d}}%
\global\long\def\tensor{\otimes}%
\global\long\def\ot{\tensor}%
\global\long\def\Rieffel#1#2{#1^{#2}}%
\global\long\def\GPsi{\Rieffel{\G}{\Psi}}%
\global\long\def\GammaPsi{\Rieffel{\Gamma}{\Psi}}%
\global\long\def\GammaPsi{\Gamma^{\Psi}}%
\global\long\def\Linfty{L^{\infty}}%
\global\long\def\Lone{L^{1}}%
\global\long\def\Ltwo{L^{2}}%
\global\long\def\LoneS{L^{1}_{\sharp}}%
\global\long\def\Ww{\mathds{W}}%
\global\long\def\wW{\text{\reflectbox{\ensuremath{\Ww}}}\:\!}%

\global\long\def\pres#1#2#3{\prescript{#1}{#2}{#3}}%

\title[Amenability via convolution semigroups]{Amenability of locally compact quantum groups via the long-time behavior
of convolution semigroups}
\author{Ami Viselter}
\address{Department of Mathematics, Faculty of Natural Sciences, University
of Haifa, 3103301 Haifa, Israel}
\email{aviselter@univ.haifa.ac.il}
\begin{abstract}
We establish a simple characterization of amenability of second countable
locally compact quantum groups $\G$ in terms of the long-time behavior
of convolution semigroups of states on $\G$. We then investigate
the deeper problem of combining amenability of $\G$ with other approximation
properties of the dual quantum group $\hat{\G}$. Specifically, we
characterize the combination of (``strong'') amenability of $\G$
with either the failure of property (T) or the presence of the Haagerup
property for $\hat{\G}$. These results---which appear to be new
even in the classical setting of locally compact groups---produce
convolution semigroups on $\G$ whose behavior is controlled both
as time goes to zero \emph{and} as time goes to infinity.
\end{abstract}

\maketitle

\section{Introduction}

Approximation and rigidity properties such as amenability \cite{Greenleaf__book},
property (T) \cite{Kazhdan__T}, and the Haagerup property \cite{Haagerup__exam_nonnucl_C_star_alg_MAP}
have been fundamental to the understanding of locally compact groups
for several decades, serving as indispensable tools in geometric group
theory, operator algebras, and ergodic theory. Among the various characterizations
of property (T) and the Haagerup property, those formulated in terms
of semigroups of positive-definite functions (or, equivalently, in
terms of conditionally negative-definite functions or in terms of
1-cocycles) \cite[Theorems~3 and 10]{Akemann_Walter__unb_neg_def_func}
stand out for their depth and utility: see, e.g., the Delorme--Guichardet
theorem and its applications in \cite[Sections~2.12 and 6.3]{Bekka_de_la_Harpe_Valette__book},
and various examples and applications throughout \cite{Cherix_Cowling_Jolissaint_Julg_Valette__book}
and in the discrete case in \cite{Brown_Ozawa__book}. Essentially,
under a countability assumption, the negation of property (T) and
the presence of the Haagerup property are equivalent to the existence
of an unbounded and a proper conditionally negative-definite function
(or 1-cocycle) on the group, respectively. These characterizations
are valuable in both directions, namely in order to establish the
property in question or to exploit it in a reinforced form.

All of the aforementioned approximation and rigidity properties of
locally compact groups (as well as others) were successfully extended
to locally compact \emph{quantum} groups in the sense of Kustermans
and Vaes \cite{Kustermans_Vaes__LCQG_C_star,Kustermans_Vaes__LCQG_von_Neumann},
see \cite{Daws_Fima_Skalski_White_Haagerup_LCQG}. Furthermore, the
notion of semigroups of positive-definite functions on groups, also
mentioned above, generalizes as follows \cite{Lindsay_Skalski__conv_semigrp_states}
(see also \cite{Lindsay_Skalski__quant_stoch_conv_cocyc_2,Lindsay_Skalski__quant_stoch_conv_cocyc_3}).
A \emph{convolution semigroup} of states on a locally compact quantum
group $\G$ is a one-parameter family $\left(\mu_{t}\right)_{t\ge0}$
of ``probability measures on $\G$'', that is, states of the universal
$C^{*}$-algebra $\CzU(\G)$, satisfying $\mu_{t+s}=\mu_{t}\star\mu_{s}$
(where $\star$ denotes the convolution), $\mu_{0}=\epsilon$, and
$w^{*}$-continuity at $t=0^{+}$. Under this definition, semigroups
of (normalized) positive-definite functions on a locally compact group
$G$ correspond precisely to convolution semigroups on its quantum
dual $\G:=\hat{G}$.

In this broader quantum context, the convolution semigroup characterizations
of property (T) and the Haagerup property were established for arbitrary
locally compact quantum groups: this is the content of the main results
of \cite[Section~4]{Skalski_Viselter__convolution_semigroups} (see
also \cite{Skalski_Viselter__generating_functionals}), which were
known earlier only for discrete quantum groups \cite{Daws_Fima_Skalski_White_Haagerup_LCQG,Daws_Skalski_Viselter__prop_T,Kyed__cohom_prop_T_QG}
(see \cite{Daws_Skalski_Viselter__prop_T} also for applications).

The initial objective of this paper was to provide a convolution semigroup
characterization of amenability parallel to those for (non-) property
(T) and the Haagerup property, as follows:
\begin{thmx}[Theorem~\ref{thm:amen_chars} and its proof]
\label{thm:A}Let $\G$ be a second countable locally compact quantum
group. Then $\G$ is amenable if and only if there exists a \emph{norm-continuous}
convolution semigroup of states $\left(\mu_{t}\right)_{t\ge0}$ on
$\G$ such that $\omega\star\mu_{t}-\mu_{t}\xrightarrow[t\to\infty]{}0$
in the $\Lone(\G)$-norm for every state $\omega\in\Lone(\G)$.
\end{thmx}

That such an asymptotically left invariant convolution semigroup can
be chosen to be norm continuous (rather than merely $w^{*}$-continuous)
is natural, because the asymptotic invariance reflects the behavior
as $t\to\infty$, whereas the type of continuity is governed by the
behavior of the convolution semigroup as $t\to0^{+}$. However, norm-continuous
convolution semigroups are analytically trivial, making it desirable
to construct asymptotically left invariant convolution semigroups
that are $w^{*}$-continuous yet not norm continuous. The immediate
obstruction to constructing such semigroups is property (T) for the
dual $\hat{\G}$, which precisely forces every $w^{*}$-convergent
net of states of $\CzU(\G)$ converging to $\epsilon$ to converge
in norm. It is therefore natural to ask whether a non-trivial, asymptotically
left invariant convolution semigroup as above exists whenever $\G$
is amenable and $\hat{\G}$ lacks property (T). While this question
is expected to be rather difficult and remains open, we resolve positively
its analog for ``strong amenability'' (that is, co-amenability of
the dual), as well as the counterpart where the failure of property
(T) is replaced by the Haagerup property, as follows. Remark that
for classical locally compact groups, the notions of amenability and
strong amenability coincide.
\begin{thmx}[Theorem~\ref{thm:strong_amen_dual_non_T_Haagerup_chars}]
\label{thm:B}Let $\G$ be a second countable locally compact quantum
group. Then $\G$ is strongly amenable and $\hat{\G}$ does not have
property (T), respectively $\hat{\G}$ has the Haagerup property,
if and only if there exists a convolution semigroup of states $\left(\mu_{t}\right)_{t\ge0}$
on $\G$ such that $\tilde{\lambda}(\mu_{t})$ has norm 1 for each
$t\ge0$ and $\tilde{\lambda}(\omega\star\mu_{t}-\mu_{t})\xrightarrow[t\to\infty]{}0$
in the operator norm for every state $\omega\in\Lone(\G)$, and $\left(\mu_{t}\right)_{t\ge0}$
is ($w^{*}$-continuous but) not norm continuous, respectively $\tilde{\lambda}(\mu_{t})\in\Cz(\hat{\G})$
for each $t>0$.
\end{thmx}

Here $\tilde{\lambda}:\CzU(\G)^{*}\to\M(\Cz(\hat{\G}))$ denotes a
``twisted'' Fourier transform. The condition $\Vert\tilde{\lambda}(\mu_{t})\Vert=1$
ensures that the asymptotic left invariance is non-trivial.

The primary strength of Theorem~\ref{thm:B} lies in constructing,
under two sets of hypotheses, convolution semigroups with desirable
features both as $t\to0^{+}$ \emph{and} as $t\to\infty$.

To the best of our knowledge, both Theorems~\ref{thm:A} and \ref{thm:B}
are new even in the classical setting of locally compact \emph{groups}
$G$. In this setting, convolution semigroups have been studied extensively
\cite{Berg_Forst__book,Heyer__book_prob_meas}, and are known to be
in one-to-one correspondence with L{\'e}vy processes on the group
\cite{Liao__book}; the quantum dual $\hat{G}$ automatically has
the Haagerup property, and fails property (T) provided that $G$ is
not discrete; and moreover, $\tilde{\lambda}$ reduces to the ordinary
Fourier transform from the measure algebra $M(G)\cong\Cz(G)^{*}$
into the multiplier algebra $\M(\Cz(\hat{G})=\CStarR(G))$ of the
reduced group $C^{*}$-algebra of $G$.

\section{Preliminaries}

\subsection{Locally compact quantum groups}

Unless mentioned otherwise, the following details are taken from \cite{Kustermans__LCQG_universal,Kustermans_Vaes__LCQG_C_star,Kustermans_Vaes__LCQG_von_Neumann};
see also \cite{Van_Daele__LCQGs}.

A \emph{locally compact quantum group} in the sense of Kustermans
and Vaes, defined in the von Neumann algebraic setting, is a pair
$\G:=\left(\Linfty(\G),\Delta\right)$, where $\Linfty(\G)$ is a
von Neumann algebra and $\Delta:\Linfty(\G)\to\Linfty(\G)\tensorn\Linfty(\G)$
is a co-multiplication: a faithful, unital, normal $*$-homomorphism
that is co-associative in the sense that $(\Delta\tensor\id)\circ\Delta=(\id\tensor\Delta)\circ\Delta$,
for which there exist left and right Haar weights (which are \emph{a
posteriori} unique up to scaling). We denote by $\Ltwo(\G)$ the GNS
Hilbert space of $\Linfty(\G)$ with respect to the left Haar weight
$\varphi$. Two other algebras pertaining to $\G$ are the reduced
$C^{*}$-algebra $\Cz(\G)$, which is WOT-dense in $\Linfty(\G)$,
and the universal $C^{*}$-algebra $\CzU(\G)$, admitting a canonical
surjective $*$-homomorphism $\CzU(\G)\to\Cz(\G)$. This allows us
to treat the predual $\Lone(\G):=\left(\Linfty(\G)\right)_{*}$ as
embedded in $\Cz(\G)^{*}$ (by restriction) and the latter as embedded
in $\CzU(\G)^{*}$ (by composing with the canonical surjection $\CzU(\G)\to\Cz(\G)$),
and we will do so henceforth. Each of $\Cz(\G)$ and $\CzU(\G)$ admits
its own version of the co-multiplication. For instance, we have the
universal co-multiplication $\Delta^{\mathrm{u}}:\CzU(\G)\to\M(\CzU(\G)\tensormin\CzU(\G))$,
where $\M(\cdot)$ stands for the multiplier algebra. Composition
with the (pre-) adjoint of the respective co-multiplication induces
the \emph{convolution} on each one of the Banach spaces $\Lone(\G),\Cz(\G)^{*},\CzU(\G)^{*}$.
For instance, given $\omega_{1},\omega_{2}\in\CzU(\G)^{*}$, we define
their convolution $\omega_{1}\star\omega_{2}\in\CzU(\G)^{*}$ to be
$(\omega_{1}\tensor\omega_{2})\circ\Delta^{\mathrm{u}}$, where, as
customary, the bounded functional $\omega_{1}\tensor\omega_{2}$ is
tacitly extended strictly from $\CzU(\G)\tensormin\CzU(\G)$ to its
multiplier algebra. With these convolutions each of $\Lone(\G),\Cz(\G)^{*},\CzU(\G)^{*}$
becomes a Banach algebra and a closed ideal in the ``bigger'' algebras.
The algebra $\left(\CzU(\G)^{*},\star\right)$ is unital; its unit
is called the co-unit and denoted by $\epsilon$.

The two factors in the ``polar decomposition'' of the antipode are
the unitary antipode $\Rant$ and the (analytic generator of) the
scaling group $\left(\tau_{t}\right)_{t\in\R}$. Each of these objects
in fact has three versions, one for each of the algebras $\Linfty(\G),\Cz(\G),\CzU(\G)$,
which fit each other as one would expect. A functional $\omega\in\CzU(\G)^{*}$
is called \emph{symmetric} if it is invariant under the (universal
version of the) unitary antipode $\Rant^{\mathrm{u}}$, i.e., if $\omega=\omega\circ\Rant^{\mathrm{u}}$. 

\emph{Duality} is the key basic feature of the theory: every locally
compact quantum group $\G$ possesses a dual locally compact quantum
group denoted by $\hat{\G}$. Objects associated with $\hat{\G}$
are decorated with hats, e.g.~$\hat{\epsilon}\in\CzU(\hat{\G})^{*}$.

The simplest examples of locally compact quantum groups are just locally
compact groups $G$. In this case, each of $\Linfty(G),\Cz(G)$ has
its usual meaning, $\CzU(G)$ equals $\Cz(G)$, the co-multiplication
is composition with the multiplication map (e.g., for $f\in\Cz(G)$,
the function $\Delta(f)\in\M(\Cz(G)\tensormin\Cz(G))\cong\Cb(G\times G)$
is given by $(\Delta(f))(t,s)=f(t\cdot s)$, $t,s\in G$), the Banach
algebra $(\Cz(G)^{*},\star)$ is just the measure algebra, and the
(unitary) antipode is given by composition with the inverse map. As
for the locally compact quantum group dual $\hat{G}$ of $G$, $\Linfty(\hat{G})$
is the group von Neumann algebra of $G$ generated by the canonical
generators $\left(\lambda_{g}\right)_{g\in G}$, $\Cz(\hat{G})$ is
the reduced group $C^{*}$-algebra of $G$, $\CzU(\hat{G})$ is the
full group $C^{*}$-algebra of $G$, the co-multiplication on $\Linfty(\hat{G})$
takes each canonical generator $\lambda_{g}$, $g\in G$, to $\lambda_{g}\tensor\lambda_{g}$,
and the Banach algebra $(\CzU(\hat{G})^{*},\star)$ is just the Fourier--Stieltjes
algebra \cite{Eymard__Fourier_alg}.

We will require the ordinary Fourier transform $\lambda$ and its
``twisted'' version $\tilde{\lambda}$. Denote by $\Ww\in\M(\CzU(\G)\tensormin\Cz(\hat{\G}))$
the version of the left regular representation of $\G$ that is ``half-universal''
on the left. For $\mu\in\CzU(\G)^{*}$ we set $\l(\mu):=\left(\mu\tensor\i\right)(\Ww)\in\M(\Cz(\hat{\G}))$
and $\tilde{\lambda}(\mu):=\hat{\tau}_{i/4}\left(\lambda(\mu)\right)\in\M(\Cz(\hat{\G}))$.
The latter equals $\big(\widetilde{R}^{(2,\varphi)}_{\overline{\mu}}\big)^{*}$
in the notation of \cite[Lemma~2.14]{Skalski_Viselter__convolution_semigroups}
(which also explains why it is well-defined). Then $\lambda,\tilde{\lambda}:(\CzU(\G)^{*},\star)\to\M(\Cz(\hat{\G}))$
are contractive Banach algebra homomorphisms, and by the above-mentioned
lemma we have $\Vert\tilde{\lambda}(\mu)\Vert\le\max(\left\Vert \lambda(\mu)\right\Vert ,\left\Vert \lambda(\overline{\mu})\right\Vert )$.
In particular, if $\mu$ is hermitian then $\Vert\tilde{\lambda}(\mu)\Vert\le\left\Vert \lambda(\mu)\right\Vert $. 

A locally compact quantum group $\G$ is called \emph{second countable}
if $\Cz(\G)$ is separable (folklore, see \cite[Definition~1.18 and Proposition~1.19]{Skalski_Viselter__convolution_semigroups}).
This is equivalent to separability of $\Ltwo(\G)$, thus also to separability
of $\Lone(\G)$.

\subsection{Convolution semigroups}
\begin{defn}[\cite{Lindsay_Skalski__conv_semigrp_states}]
A \emph{convolution semigroup} of contractive positive linear functionals
on a locally compact quantum group $\G$ (more explicitly, on $\CzU(\G)$)
is a one-parameter family $\left(\mu_{t}\right)_{t\ge0}$ of contractive
positive linear functionals on $\CzU(\G)$ satisfying $\mu_{t+s}=\mu_{t}\star\mu_{s}$
for all $t,s\ge0$ and $\mu_{0}=\epsilon$. We say that $\left(\mu_{t}\right)_{t\ge0}$
is \emph{$w^{*}$-continuous} if $\mu_{t}\xrightarrow[t\to0^{+}]{}\epsilon$
in the $w^{*}$-topology, i.e.~$\mu_{t}(a)\xrightarrow[t\to0^{+}]{}\epsilon(a)$
for all $a\in\CzU(\G)$. We say that $\left(\mu_{t}\right)_{t\ge0}$
is \emph{symmetric} if each of its elements is symmetric.
\end{defn}

The analytically trivial convolution semigroups of states on $\G$
are those that are \emph{norm continuous}, i.e.~$\mu_{t}\xrightarrow[t\to0^{+}]{}\epsilon$
in the norm topology of $\CzU(\G)^{*}$. By \cite[Proposition~2.3]{Lindsay_Skalski__conv_semigrp_states}
and \cite[Section~6]{Lindsay_Skalski__quant_stoch_conv_cocyc_3} (see
also \cite[Remark~2.10]{Skalski_Viselter__generating_functionals}),
these are precisely the convolution semigroups of the form $\left(\exp_{\star}(t\gamma):=\sum^{\infty}_{n=0}\frac{t^{n}}{n!}\gamma^{\star n}\right)_{t\ge0}$
(exponentiation in the unital Banach algebra $\left(\CzU(\G)^{*},\star\right)$),
where $\gamma\in\CzU(\G)^{*}$ is a normalized, hermitian, conditionally
positive functional; here conditional positivity means positivity
on $\ker\epsilon$, that is: $\gamma(a)\ge0$ for all $a\in\CzU(\G)_{+}\cap\ker\epsilon$,
and being normalized means that the strict extension of $\gamma$
to $\M(\CzU(\G))$ vanishes at the unit $\one$. In turn, these functionals
$\gamma$ are precisely those of the form $\gamma=s(\mu-\epsilon)$,
where $\mu$ is a state of $\CzU(\G)$ and $s\ge0$. Classically,
these convolution semigroups are the families of distributions of
compound Poisson processes.

Of key importance to our methods is the following. For $\mu\in\CzU(\G)^{*}$,
define the convolution operators $L^{\mathrm{u}}_{\mu},R^{\mathrm{u}}_{\mu}:\CzU(\G)\to\CzU(\G)$
by $R^{\mathrm{u}}_{\mu}:=(\mu\tensor\i)\circ\Delta^{\mathrm{u}}$
and $L^{\mathrm{u}}_{\mu}:=(\i\tensor\mu)\circ\Delta^{\mathrm{u}}$,
and note that these operators commute and $\left\Vert L^{\mathrm{u}}_{\mu}\right\Vert =\left\Vert R^{\mathrm{u}}_{\mu}\right\Vert =\left\Vert \mu\right\Vert $.
\begin{thm}[{\cite[parts of Theorem~3.2, Remark~3.3 and Theorem~3.4]{Skalski_Viselter__convolution_semigroups}}]
\label{thm:SV_JMPA_3.2_3.3_3.4}Let $\G$ be a locally compact quantum
group. Then there exists a one-to-one correspondence between $w^{*}$-continuous
convolution semigroups of contractive positive functionals $\left(\mu_{t}\right)_{t\ge0}$
on $\G$ and $C_{0}$-semigroups $\left(T^{\mathrm{u}}_{t}\right)_{t\ge0}$
of contractive, completely positive maps on $\CzU(\G)$ that commute
with the operators $L^{\mathrm{u}}_{\nu}$, $\nu\in\CzU(\G)^{*}$.
The correspondence is given by $T^{\mathrm{u}}_{t}:=R^{\mathrm{u}}_{\mu_{t}}$,
$t\ge0$. In addition, $\left(\mu_{t}\right)_{t\ge0}$ is norm continuous
if and only if $\left(R^{\mathrm{u}}_{\mu_{t}}\right)_{t\ge0}$ is.
If $\left(\mu_{t}\right)_{t\ge0}$ is symmetric, then $(\tilde{\lambda}(\mu_{t}))_{t\ge0}$
is a $C_{0}$-semigroup of contractive selfadjoint operators on $\Ltwo(\G)$,
and $\left(\mu_{t}\right)_{t\ge0}$ is norm continuous if and only
if $(\tilde{\lambda}(\mu_{t}))_{t\ge0}$ is. 
\end{thm}

Remark that in the symmetric case we have, in the notation of \cite{Skalski_Viselter__convolution_semigroups},
$\tilde{\lambda}(\mu_{t})=\big(\widetilde{R}^{(2,\varphi)}_{\overline{\mu_{t}}}\big)^{*}=\widetilde{R}^{(2,\varphi)}_{\mu_{t}}$.
Also, in this case, more can be said about $(\tilde{\lambda}(\mu_{t}))_{t\ge0}$
and about the associated Dirichlet form. We leave out the full details
of \cite[Theorem~3.4, correspondence between (a), (c) and (d)]{Skalski_Viselter__convolution_semigroups}
to make our paper more easy to read.

\subsection{Amenability, the Haagerup property and property (T)}
\begin{defn}[\cite{Bedos_Tuset_2003}, see also \cite{Voiculescu__amen_Kac_alg,Enock_Schwartz__amenable_Kac_alg,Desmedt_Quaegebeur_Vaes}]
\label{def:amen_strong_amen}\mbox{}
\begin{itemize}
\item We say that $\G$ is\emph{ amenable} if there exists a net $\left(\omega_{i}\right)_{i\in\mathcal{I}}$
of states in $\Lone(\G)$ such that 
\begin{equation}
\left\Vert \omega\star\omega_{i}-\omega_{i}\right\Vert \xrightarrow[i\in\mathcal{I}]{}0\qquad(\forall\text{ state }\omega\in\Lone(\G)).\label{eq:amenability}
\end{equation}
\item We say that $\G$ is \emph{strongly amenable}\footnote{While the terminology ``$\G$ is strongly amenable'' is outdated
as nowadays most authors write ``$\hat{\G}$ is co-amenable'', we
retain both terms side-by-side for ease of exposition.}, or that $\hat{\G}$ is \emph{co-amenable}, if the Banach algebra
$\big(\Cz(\hat{\G})^{*},\star\big)$ is unital; this is equivalent
to the equality of the $C^{*}$-algebras $\CzU(\hat{\G})$ and $\Cz(\hat{\G})$,
i.e., to the canonical surjection $\CzU(\hat{\G})\to\Cz(\hat{\G})$
being injective.
\end{itemize}
\end{defn}

Each of the two notions of amenability has many equivalent characterizations,
most of which we will not spell out here. As the terminology suggests,
strong amenability implies amenability \cite[Theorem~3.2]{Bedos_Tuset_2003},
and it has long been open whether the converse is also true, as it
is for locally compact groups as well as for discrete quantum groups
\cite{Blanchard_Vaes,Tomatsu__amenable_discrete}.
\begin{rem}
\label{rem:amen_Lone_CzU_star}We make two standard observations stemming
from $\Lone(\G)$ being an ideal in the Banach algebra $(\CzU(\G)^{*},\star)$.
First, if a net $\left(\omega_{i}\right)_{i\in\mathcal{I}}$ of states
in $\CzU(\G)^{*}$ (rather than $\Lone(\G)$) satisfies (\ref{eq:amenability}),
then $\G$ is still amenable, because $\left(\omega_{i}\right)_{i\in\mathcal{I}}$
may be replaced by $\left(\omega_{i}\star\omega_{0}\right)_{i\in\mathcal{I}}$
for an arbitrary state $\omega_{0}\in\Lone(\G)$. Second, if a net
$\left(\omega_{i}\right)_{i\in\mathcal{I}}$ of states in $\Lone(\G)$
(or, more generally, in $\CzU(\G)^{*}$) satisfies (\ref{eq:amenability}),
then the same limiting condition actually holds for every state $\omega\in\CzU(\G)^{*}$,
because for a state $\omega_{0}\in\Lone(\G)$ we then have $\left\Vert \omega_{0}\star\omega_{i}-\omega_{i}\right\Vert \xrightarrow[i\in\mathcal{I}]{}0$,
thus $\left\Vert \omega\star\omega_{0}\star\omega_{i}-\omega\star\omega_{i}\right\Vert \xrightarrow[i\in\mathcal{I}]{}0$,
and also $\left\Vert \omega\star\omega_{0}\star\omega_{i}-\omega_{i}\right\Vert \xrightarrow[i\in\mathcal{I}]{}0$
as $\omega\star\omega_{0}\in\Lone(\G)$, which put together yield
that $\left\Vert \omega\star\omega_{i}-\omega_{i}\right\Vert \xrightarrow[i\in\mathcal{I}]{}0$.
\end{rem}

\begin{rem}
\label{rem:amen_symm}The net $\left(\omega_{i}\right)_{i\in\mathcal{I}}$
in the definition of amenability can be assumed to be symmetric. Indeed,
this is essentially noted in \cite[proof of Proposition~3]{Desmedt_Quaegebeur_Vaes}:
if a net $\left(\omega_{i}\right)_{i\in\mathcal{I}}$ of states in
$\Lone(\G)$ satisfies $\left\Vert \omega\star\omega_{i}-\omega_{i}\right\Vert \xrightarrow[i\in\mathcal{I}]{}0$
for every state $\omega\in\Lone(\G)$, then $\left(\omega_{i}':=\omega_{i}\star(\omega_{i}\circ\Rant)\right)_{i\in\mathcal{I}}$
is a net of symmetric states in $\Lone(\G)$ satisfying $\left\Vert \omega\star\omega_{i}'-\omega_{i}'\right\Vert \xrightarrow[i\in\mathcal{I}]{}0$
(as well as $\left\Vert \omega_{i}'\star\omega-\omega_{i}'\right\Vert \xrightarrow[i\in\mathcal{I}]{}0$)
for every state $\omega\in\Lone(\G)$.
\end{rem}

\begin{rem}
When $\G$ is second countable and amenable, the net $\left(\omega_{i}\right)_{i\in\mathcal{I}}$
in the definition of amenability may be chosen to be a sequence.
\end{rem}

\begin{defn}[{\cite[Definition~6.1 and the next paragraph]{Daws_Fima_Skalski_White_Haagerup_LCQG}}]
We say that a locally compact quantum group $\G$ has:
\begin{itemize}
\item \emph{property (T)} if each unitary representation of $\G$ that has
almost-invariant vectors has a non-zero invariant vector;
\item the \emph{Haagerup property} if $\G$ admits a mixing unitary representation
with almost-invariant vectors.
\end{itemize}
\end{defn}

We do not explain the notions appearing in this definition, and refer
to \cite{Daws_Fima_Skalski_White_Haagerup_LCQG} for further details.
By \cite[Theorem~6.1]{Daws_Skalski_Viselter__prop_T}, the dual $\hat{\G}$
has property (T) if and only if every net of states of $\CzU(\G)$
that converges to $\epsilon$ in the $w^{*}$-topology actually converges
in norm; and by \cite[Proposition~4.9]{Skalski_Viselter__convolution_semigroups},
$\hat{\G}$ has the Haagerup property if and only if there exists
a net $\left(\mu_{i}\right)_{i\in\mathcal{I}}$ of states of $\CzU(\G)$
such that $\left(\lambda(\mu_{i})\right)_{i\in\mathcal{I}}$, respectively
$(\tilde{\lambda}(\mu_{i}))_{i\in\mathcal{I}}$, is an approximate
identity for the $C^{*}$-algebra $\Cz(\hat{\G})$. One might as well
take these as the definitions of property (T) and the Haagerup property.
Together with Skalski we proved in \cite{Skalski_Viselter__convolution_semigroups}
far-reaching improvements of these characterizations involving convolution
semigroups. See more in the sequel.

\subsection{Perturbations of $C_{0}$-semigroups by bounded operators}

The following two results are well-known. We state them here in the
form we need for the convenience of the reader.
\begin{thm}[{\cite[Theorems~III.1.3, III.1.10 and Corollary~5.8]{Engel_Nagel__one_param_semigr_lin_evo_eq}}]
\label{thm:C0_bdd_perturb}Let $X$ be a Banach space, let $\left(T_{1}(t)\right)_{t\ge0}$
be a $C_{0}$-semigroup on $X$ with generator $A_{1}$, and let $A_{2}\in B(X)$.
Then $A_{1}+A_{2}$ generates a $C_{0}$-semigroup $\left(S(t)\right)_{t\ge0}$
on $X$. Furthermore:
\begin{itemize}
\item (The Dyson--Phillips theorem) Recursively define a family $\left(S_{n}(t)\right)_{t\ge0,n\in\Z_{+}}$
in $B(X)$ by $S_{0}(t):=T_{1}(t)$ and 
\begin{equation}
S_{n+1}(t):=\mathrm{SOT}\text{-}\int^{t}_{0}T_{1}(t-s)A_{2}S_{n}(s)\d s=\mathrm{SOT}\text{-}\int^{t}_{0}S_{n}(t-s)A_{2}T_{1}(s)\d s\label{eq:Dyson_Phillips}
\end{equation}
(the maps $S_{n}(\cdot)$ and the integrands are SOT-continuous).
Then for every $t\ge0$ we have $S(t)=\sum^{\infty}_{n=0}S_{n}(t)$,
where the series converges in the (operator) norm of $B(X)$. 
\item (The Trotter product formula) If the $C_{0}$-semigroups $\left(T_{1}(t)\right)_{t\ge0}$,
$\left(T_{2}(t):=e^{tA_{2}}\right)_{t\ge0}$ are contractive, then
\[
S(t)=\mathrm{SOT}\text{-}\lim_{n\to\infty}\left(T_{1}(\frac{t}{n})T_{2}(\frac{t}{n})\right)^{n}=\mathrm{SOT}\text{-}\lim_{n\to\infty}\left(T_{2}(\frac{t}{n})T_{1}(\frac{t}{n})\right)^{n}\qquad(\forall t\ge0).
\]
\end{itemize}
\end{thm}

The first equality in (\ref{eq:Dyson_Phillips}) is part of the standard
statement of the theorem. The second equality is apparently well-known
to $C_{0}$-semigroup experts, but is not written explicitly in any
text we know, although \cite[Exercise~(3)~(ii) in Section~III.1]{Engel_Nagel__one_param_semigr_lin_evo_eq}
comes close. We provide a direct proof for completeness. Suppose that
(\ref{eq:Dyson_Phillips}) holds true for some $n\in\Z_{+}$ and all
$t\ge0$. Then, for all $t\ge0$, interpreting all integrals in the
SOT, we obtain 
\[
S_{n+2}(t)=\int^{t}_{0}T_{1}(t-s)A_{2}S_{n+1}(s)\d s=\int^{t}_{0}\left(\int^{s}_{0}T_{1}(t-s)A_{2}S_{n}(s-u)A_{2}T_{1}(u)\d u\right)\dd s.
\]
Applying the change of variable $x:=s-u$, $y:=u$ (justified, e.g.,
by viewing the integrals in the WOT) gives 
\[
S_{n+2}(t)=\int^{t}_{0}\left(\int^{t-y}_{0}T_{1}(t-y-x)A_{2}S_{n}(x)A_{2}T_{1}(y)\d x\right)\dd y=\int^{t}_{0}S_{n+1}(t-y)A_{2}T_{1}(y)\d y,
\]
as desired.

\section{Approximation properties via convolution semigroups}

The starting point of this work is the following result of Kalantar,
Neufang and Ruan, generalizing to locally compact quantum groups a
classical and important result proved independently by Rosenblatt
for general locally compact groups \cite[Theorem~1.10]{Rosenblatt__ergod_mix_random_walk_lc_grp}
and by Kaimanovich and Vershik for discrete groups \cite[Theorem~4.3]{Kaimanovich_Vershik__random_walks_discrete_groups_bdry_entropy},
and later extended to discrete Kac algebras by Vaes \cite[Lemma~7.1]{Vaes_strict_out_act}.
\begin{thm}[{\cite[Theorem~4.1]{Kalantar_Neufang_Ruan__Poisson_bdry_LCQG}}]
\label{thm:KNR__amen_by_single_state}A second countable locally
compact quantum group $\G$ is amenable if and only if there exists
a state $\mu\in\CzU(\G)^{*}$, which may even be taken from $\Lone(\G)$,
such that $\left\Vert \omega\star\mu^{\star n}-\mu^{\star n}\right\Vert \xrightarrow[n\to\infty]{}0$
for every state $\omega\in\Lone(\G)$.
\end{thm}

\begin{rem}
\label{rem:amen_by_single_state_symm}By Remark~\ref{rem:amen_symm},
the state $\mu$ in the statement of Theorem~\ref{thm:KNR__amen_by_single_state}
may be chosen symmetric, as a simple inspection of the proof of that
result shows.
\end{rem}

Theorem~\ref{thm:KNR__amen_by_single_state} was used by Rosenblatt
and by Kaimanovich and Vershik to verify a conjecture of Furstenberg
saying that every amenable second countable locally compact group
admits a probability measure with the Liouville property, namely one
that possesses no harmonic $\Linfty$-functions except the constant
ones. This was extended to locally compact quantum groups in \cite[Theorem~4.2, see also Remark~4.3]{Kalantar_Neufang_Ruan__Poisson_bdry_LCQG}.

We are interested in strengthening Theorem~\ref{thm:KNR__amen_by_single_state}
by replacing the discrete-time convolution semigroup $\left(\mu^{\star n}\right)^{\infty}_{n=0}$
by a continuous-time one $\left(\mu_{t}\right)_{t\ge0}$.
\begin{thm}
\label{thm:amen_chars}Let $\G$ be a second countable locally compact
quantum group. Then the following conditions are equivalent:
\begin{enumerate}
\item \label{enu:amen_chars__1}$\G$ is amenable, i.e., there exists a
sequence $\left(\omega_{n}\right)^{\infty}_{n=1}$ of states in $\Lone(\G)$
such that
\[
\left\Vert \omega\star\omega_{n}-\omega_{n}\right\Vert \xrightarrow[n\to\infty]{}0\qquad(\forall\text{ state }\omega\in\Lone(\G));
\]
\item \label{enu:amen_chars__2}there exists a $w^{*}$-continuous convolution
semigroup of states $\left(\mu_{t}\right)_{t\ge0}$ on $\G$ such
that 
\begin{equation}
\left\Vert \omega\star\mu_{t}-\mu_{t}\right\Vert \xrightarrow[t\to\infty]{}0\qquad(\forall\text{ state }\omega\in\Lone(\G));\label{eq:amen_chars__2}
\end{equation}
\item \label{enu:amen_chars__3}there exists a \emph{symmetric, norm-continuous}
convolution semigroup of states $\left(\mu_{t}\right)_{t\ge0}$ on
$\G$ satisfying (\ref{eq:amen_chars__2}).
\end{enumerate}
\end{thm}

\begin{proof}
The implications \ref{enu:amen_chars__3}$\implies$\ref{enu:amen_chars__2}$\implies$\ref{enu:amen_chars__1}
are trivial. 

\ref{enu:amen_chars__1}$\implies$\ref{enu:amen_chars__3}. Since
$\G$ is second countable we may apply Theorem~\ref{thm:KNR__amen_by_single_state}
and Remark~\ref{rem:amen_by_single_state_symm} to obtain a symmetric
state $\mu\in\Lone(\G)$ such that $\left\Vert \omega\star\mu^{\star n}-\mu^{\star n}\right\Vert \xrightarrow[n\to\infty]{}0$
for every state $\omega\in\Lone(\G)$. We will prove that the symmetric,
norm-continuous convolution semigroup of states $\left(\mu_{t}\right)_{t\ge0}$
generated by the symmetric, normalized, hermitian, conditionally positive
functional $\mu-\epsilon\in\CzU(\G)^{*}$, namely 
\[
\mu_{t}:=\exp_{\star}\left(t\left(\mu-\epsilon\right)\right)=e^{-t}\exp_{\star}\left(t\mu\right)=e^{-t}\sum^{\infty}_{n=0}\frac{t^{n}}{n!}\mu^{\star n}\in\CzU(\G)^{*}\qquad(t\ge0),
\]
satisfies (\ref{eq:amen_chars__2}). Indeed, fix a state $\omega\in\Lone(\G)$
and $\e>0$. Choose $n_{0}\in\N$ such that $\left\Vert \omega\star\mu^{\star n}-\mu^{\star n}\right\Vert \le\e$
for all $n>n_{0}$. For every $t\ge0$ we now have
\[
\omega\star\mu_{t}-\mu_{t}=e^{-t}\sum^{\infty}_{n=0}\frac{t^{n}}{n!}\left(\omega\star\mu^{\star n}-\mu^{\star n}\right).
\]
Thus, 
\[
\begin{split}\left\Vert \omega\star\mu_{t}-\mu_{t}\right\Vert  & \le e^{-t}\sum^{\infty}_{n=0}\frac{t^{n}}{n!}\left\Vert \omega\star\mu^{\star n}-\mu^{\star n}\right\Vert \\
 & \le e^{-t}\left(\sum^{n_{0}}_{n=0}2\frac{t^{n}}{n!}+\sum^{\infty}_{n=n_{0}+1}\e\frac{t^{n}}{n!}\right)\le e^{-t}\sum^{n_{0}}_{n=0}2\frac{t^{n}}{n!}+\e\xrightarrow[t\to\infty]{}\e,
\end{split}
\]
where in the second inequality we used the fact that $\left\Vert \omega\right\Vert =1$.
This proves the assertion.
\end{proof}

\begin{rem}
In the proof of \ref{enu:amen_chars__1}$\implies$\ref{enu:amen_chars__3}
we have not used the fact that $\mu\in\Lone(\G)$; each $\mu\in\CzU(\G)^{*}$
with the same properties would do. However, when $\mu$ is chosen
from $\Lone(\G)$, all states in the resulting convolution semigroup
$\left(\mu_{t}\right)_{t\ge0}$ do \emph{not} belong to $\Lone(\G)$
unless $\epsilon\in\Lone(\G)$, i.e., unless $\G$ is discrete \cite[Proposition~4.1]{Runde__charac_compact_discr_QG},
since for every $t\ge0$ we have $\mu_{t}\in e^{-t}(\epsilon+\Lone(\G))\subseteq\CzU(\G)^{*}$.
This is not surprising: having $\mu_{t}\in\Lone(\G)$ for each $t>0$
when the convolution semigroup $\left(\mu_{t}\right)_{t\ge0}$ is
norm continuous entails that the closed ideal $\Lone(\G)$ of the
Banach algebra $(\CzU(\G)^{*},\star)$ contains the (co-) unit $\epsilon$,
equivalently that $\G$ is discrete. In the general setting of a $w^{*}$-continuous
(but not necessarily norm-continuous) convolution semigroup of states
$\left(\mu_{t}\right)_{t\ge0}$, having $\mu_{t}\in\Lone(\G)$ for
each $t>0$ implies that $\epsilon\in\Cz(\G)^{*}$, thus $\G$ is
co-amenable. See Open Question~\ref{ques:strong_amen_in_terms_of_hat_G}
in the end of the paper.
\end{rem}

The fact that the convolution semigroup constructed in the proof of
Theorem~\ref{thm:amen_chars} is norm continuous calls for a further
improvement, namely:
\begin{question}
\label{ques:amen}Under which conditions does amenability of $\G$
imply the existence of a $w^{*}$-continuous, but \emph{not norm-continuous},
convolution semigroup of states $\left(\mu_{t}\right)_{t\ge0}$ on
$\G$ satisfying (\ref{eq:amen_chars__2}) of Theorem~\ref{thm:amen_chars}?
\end{question}

The obvious obstruction is $\hat{\G}$ having property (T), which
means that when a net of states of $\CzU(\G)$ converges to $\epsilon$
in the $w^{*}$-topology, it actually converges in norm \cite[Theorem~6.1]{Daws_Skalski_Viselter__prop_T}.

Question~\ref{ques:amen} seems to be fairly difficult because operator-theoretic
techniques are not readily applicable, and we were not able to solve
this question aside from the trivial, namely the co-commutative, case,
which we now explain. So suppose that $\G$ equals $\hat{G}$ for
a second countable locally compact group $G$ that does not have property
(T). Since $G$ does not have property (T), there exists a $w^{*}$-continuous,
but not norm-continuous, convolution semigroup of states $\left(\mu^{1}_{t}\right)_{t\ge0}$
on $\G=\hat{G}$ by \cite[Theorem~3]{Akemann_Walter__unb_neg_def_func}
(or by the far more general \cite[Theorem~4.6]{Skalski_Viselter__convolution_semigroups}).
Since $\G=\hat{G}$ is (trivially) amenable, there exists a norm-continuous
convolution semigroup of states $\left(\mu^{2}_{t}\right)_{t\ge0}$
on $\G$ such that $\left\Vert \omega\star\mu^{2}_{t}-\mu^{2}_{t}\right\Vert \xrightarrow[t\to\infty]{}0$
for each state $\omega\in\Lone(\G)$ by Theorem~\ref{thm:amen_chars}.
By co-commutativity of $\G$, the family $\left(\mu_{t}:=\mu^{1}_{t}\star\mu^{2}_{t}\right)_{t\ge0}$
is a $w^{*}$-continuous convolution semigroup of states on $\G$
that is not norm continuous and that satisfies (\ref{eq:amen_chars__2})
of Theorem~\ref{thm:amen_chars}.

To give concrete examples coming from groups for convolution semigroups
with the properties listed in Question~\ref{ques:amen}, the following
observation, which must be folklore and should be compared with the
subtle relationship between (say, left) invariant vs.~topological
invariant means on locally compact groups \cite[Theorem~2.2.1 and its proof]{Greenleaf__book},
is useful.
\begin{rem}
\label{rem:top_vs_transl_asymp_left_inv}Let $G$ be a locally compact
group. For $\omega\in\Lone(G)$ and $g\in G$ write $L_{g}\omega:=\delta_{g}\star\omega$
for the left $g$-translate of $\omega$ given by $(L_{g}\omega)(h):=\omega(g^{-1}h)$,
$h\in G$. Then a sequence $\left(\omega_{n}\right)^{\infty}_{n=1}$
of probability measures in $\Lone(G)$ satisfies (\ref{eq:amenability}),
that is, $\left\Vert \omega\star\omega_{n}-\omega_{n}\right\Vert \xrightarrow[n\to\infty]{}0$
for each probability measure $\omega\in\Lone(\G)$, if and only if
for every $g\in G$ we have $\left\Vert L_{g}\omega_{n}-\omega_{n}\right\Vert \xrightarrow[n\to\infty]{}0$.
Indeed, the forward direction ($\implies$) follows from Remark~\ref{rem:amen_Lone_CzU_star}.
For the converse implication ($\impliedby$), let $\omega\in\Lone(G)$
be a probability measure. Then by \cite[Appendix~4]{Folland__course_in_abs_harm_anal_book_2nd_ed},
for each $\rho\in\Lone(G)$ we have $\omega\star\rho=\int_{G}\omega(g)L_{g}\rho\d g$,
where the $\Lone(G)$-valued integral is in the weak sense. Therefore,
for every $n\in\N$ we get $\omega\star\omega_{n}-\omega_{n}=\int_{G}\omega(g)\left(L_{g}\omega_{n}-\omega_{n}\right)\dd g$,
thus $\left\Vert \omega\star\omega_{n}-\omega_{n}\right\Vert \le\int_{G}\left|\omega(g)\right|\left\Vert L_{g}\omega_{n}-\omega_{n}\right\Vert \dd g$.
Noting that $\left|\omega(g)\right|\left\Vert L_{g}\omega_{n}-\omega_{n}\right\Vert \le2\left|\omega(g)\right|$
for each $n\in\N$ and $g\in G$, we can apply Lebesgue's dominated
convergence theorem to deduce that $\int_{G}\left|\omega(g)\right|\left\Vert L_{g}\omega_{n}-\omega_{n}\right\Vert \dd g\xrightarrow[n\to\infty]{}0$,
hence that $\left\Vert \omega\star\omega_{n}-\omega_{n}\right\Vert \xrightarrow[n\to\infty]{}0$,
as claimed. Remark that the above argument works neither for nets
instead of sequences nor for sequences in $\Cz(G)^{*}$ instead of
in $\Lone(G)$.
\end{rem}

\begin{example}
Let $\G:=\R$. The Brownian semigroup given by $\frac{\dd\mu^{\mathrm{B}}_{t}}{\dd\text{(Lebesgue)}}=f^{\mathrm{B}}_{t}$
with $f^{\mathrm{B}}_{t}(x)=\frac{1}{\sqrt{4\pi t}}e^{-x^{2}/(4t)}$
(equivalently, $\widehat{\mu_{t}}=e^{-t\left(\cdot\right)^{2}}$)
and the Cauchy semigroup given by $\frac{\dd\mu^{\mathrm{C}}_{t}}{\dd\text{(Lebesgue)}}=f^{\mathrm{C}}_{t}$
with $f^{\mathrm{C}}_{t}(x)=\frac{t}{\pi\left(t^{2}+x^{2}\right)}$
(equivalently, $\widehat{\mu_{t}}=e^{-t\left|\cdot\right|}$), which
are symmetric and not norm continuous, satisfy (\ref{eq:amen_chars__2})
of Theorem~\ref{thm:amen_chars}. We prove that using Remark~\ref{rem:top_vs_transl_asymp_left_inv}.
Fix $i\in\left\{ \mathrm{B},\mathrm{C}\right\} $ and $0<x_{0}\in\R$
(negative values of $x_{0}$ are treated similarly). Then for $t>0$,
the function $L_{x_{0}}f^{i}_{t}-f^{i}_{t}$ vanishes only at $\frac{x_{0}}{2}$.
Consequently, using the fact that $\mu^{i}_{t}$ is a probability
measure, we deduce that 
\[
\begin{split}\left\Vert L_{x_{0}}f^{i}_{t}-f^{i}_{t}\right\Vert  & =\int_{\R}\left|f^{i}_{t}(x-x_{0})-f^{i}_{t}(x)\right|\dd x\\
 & =\int^{x_{0}/2}_{-\infty}\left(f^{i}_{t}(x)-f^{i}_{t}(x-x_{0})\right)\dd x+\int^{\infty}_{x_{0}/2}\left(f^{i}_{t}(x-x_{0})-f^{i}_{t}(x)\right)\dd x\\
 & =2\int^{x_{0}/2}_{-\infty}f^{i}_{t}(x)\d x+2\int^{\infty}_{x_{0}/2}f^{i}_{t}(x-x_{0})\d x-2.
\end{split}
\]
By symmetry we have $\int^{\infty}_{x_{0}/2}f^{i}_{t}(x-x_{0})\d x=\int^{x_{0}/2}_{-\infty}f^{i}_{t}(x)\d x$.
To conclude, for all $t>0$ we have 
\[
\left\Vert L_{x_{0}}f^{i}_{t}-f^{i}_{t}\right\Vert =4\int^{x_{0}/2}_{-\infty}f^{i}_{t}(x)\d x-2.
\]
This indeed goes to zero as $t\to\infty$:
\begin{itemize}
\item When $i=\mathrm{B}$ we have $\int^{x_{0}/2}_{-\infty}f^{\mathrm{B}}_{t}(x)\d x=\frac{1}{\sqrt{\pi}}\int^{x_{0}/(4\sqrt{t})}_{-\infty}e^{-y^{2}}\d y\xrightarrow[t\to\infty]{}\frac{1}{\sqrt{\pi}}\int^{0}_{-\infty}e^{-y^{2}}\d y=\frac{1}{2}$.
\item When $i=\mathrm{C}$ we have $\int^{x_{0}/2}_{-\infty}f^{\mathrm{C}}_{t}(x)\d x=\frac{1}{\pi}\int^{x_{0}/(2t)}_{-\infty}\frac{1}{1+y^{2}}\d y\xrightarrow[t\to\infty]{}\frac{1}{\pi}\int^{0}_{-\infty}\frac{1}{1+y^{2}}\d y=\frac{1}{2}$.
\end{itemize}
\end{example}

Although we could not solve the amenability-related Question~\ref{ques:amen},
we were able to solve its analog for strong amenability (i.e.~co-amenability
of the dual quantum group), to which operator-theoretic methods apply
more naturally. We first present the following characterization of
strong amenability, which is analogous to (and partly follows from)
Theorem~\ref{thm:amen_chars}.
\begin{thm}
\label{thm:strong_amen_chars}Let $\G$ be a second countable locally
compact quantum group. Then the following conditions are equivalent:
\begin{enumerate}
\item \label{enu:strong_amen_chars__1}$\G$ is strongly amenable (that
is, $\hat{\G}$ is co-amenable);
\item \label{enu:strong_amen_chars__2}there exists a net $\left(X_{i}\right)_{i\in\mathcal{I}}$
of operators of norm $1$ in $\M(\Cz(\hat{\G}))$ such that
\[
\left\Vert \lambda(\omega)X_{i}-X_{i}\right\Vert \xrightarrow[i\in\mathcal{I}]{}0\qquad(\forall\text{ state }\omega\in\Lone(\G));
\]
\item \label{enu:strong_amen_chars__3}there exists a $w^{*}$-continuous
convolution semigroup of states $\left(\mu_{t}\right)_{t\ge0}$ on
$\G$ such that
\begin{equation}
\left\Vert \lambda(\omega\star\mu_{t}-\mu_{t})\right\Vert \xrightarrow[t\to\infty]{}0\qquad(\forall\text{ state }\omega\in\Lone(\G))\label{eq:strong_amen_chars__3}
\end{equation}
and such that $\left\Vert \lambda(\mu_{t})\right\Vert =1$ for each
$t\ge0$;
\item \label{enu:strong_amen_chars__4}condition \ref{enu:strong_amen_chars__3}
holds with $\left(\mu_{t}\right)_{t\ge0}$ being symmetric and norm
continuous.
\end{enumerate}
The equivalence continues to hold when $\lambda$ is replaced by $\tilde{\lambda}$
throughout \ref{enu:strong_amen_chars__2},\ref{enu:strong_amen_chars__3},
and \ref{enu:strong_amen_chars__4}.

\end{thm}

\begin{proof}
The implications \ref{enu:strong_amen_chars__4}$\implies$\ref{enu:strong_amen_chars__3}$\implies$\ref{enu:strong_amen_chars__2}
are immediate. It remains to prove that \ref{enu:strong_amen_chars__2}$\implies$\ref{enu:strong_amen_chars__1}
and \ref{enu:strong_amen_chars__1}$\implies$\ref{enu:strong_amen_chars__4}.

\ref{enu:strong_amen_chars__2}$\implies$\ref{enu:strong_amen_chars__1}.
Assume a net $\left(X_{i}\right)_{i\in\mathcal{I}}$ as in \ref{enu:strong_amen_chars__2}
is given. Then for every $\omega\in\Lone(\G)$ we have $\left\Vert \lambda(\omega)X_{i}-\omega(\one)X_{i}\right\Vert \xrightarrow[i\in\mathcal{I}]{}0$.
Define a functional $\mathtt{e}:\lambda\left(\Lone(\G)\right)\to\C$
by $\mathtt{e}(\lambda(\omega)):=\omega(\one)$ for $\omega\in\Lone(\G)$.
Then $\mathtt{e}$ is well-defined and contractive, because for each
$\omega$, using the fact that $\left(X_{i}\right)_{i\in\mathcal{I}}$
consists of norm one operators, we obtain
\[
\left|\omega(\one)\right|=\left|\omega(\one)X_{i}\right|\le\left\Vert \lambda(\omega)X_{i}-\omega(\one)X_{i}\right\Vert +\left\Vert \lambda(\omega)X_{i}\right\Vert \le\left\Vert \lambda(\omega)X_{i}-\omega(\one)X_{i}\right\Vert +\left\Vert \lambda(\omega)\right\Vert \xrightarrow[i\in\mathcal{I}]{}\left\Vert \lambda(\omega)\right\Vert .
\]
So $\mathtt{e}$ extends to all of $\Cz(\hat{\G})$, proving that
$\hat{\G}$ is co-amenable.

\ref{enu:strong_amen_chars__1}$\implies$\ref{enu:strong_amen_chars__4}.
Since $\hat{\G}$ is co-amenable, $\G$ is amenable. By Theorem~\ref{thm:amen_chars},
there is a symmetric, norm-continuous convolution semigroup of states
$\left(\mu_{t}\right)_{t\ge0}$ on $\G$ such that $\left\Vert \omega\star\mu_{t}-\mu_{t}\right\Vert \xrightarrow[t\to\infty]{}0$
for each state $\omega\in\Lone(\G)$. Since $\lambda:\Lone(\G)\to\M(\Cz(\hat{\G}))$
is contractive, (\ref{eq:strong_amen_chars__3}) holds. Co-amenability
of $\hat{\G}$ implies that for every $t\ge0$ we have $\left\Vert \lambda(\mu_{t})\right\Vert =1$,
because $1=\left|\mu_{t}(\one)\right|=\left|\hat{\epsilon}\left(\lambda(\mu_{t})\right)\right|\le\left\Vert \lambda(\mu_{t})\right\Vert \le\left\Vert \mu_{t}\right\Vert =1$.

Replacing $\lambda$ by $\tilde{\lambda}$ leaves all arguments unchanged.
\end{proof}

The following theorem is the main result of this paper. The primary
challenge is to combine two convolution semigroups having two distinct
desirable properties into a single convolution semigroup admitting
both properties. Note that one of these properties has to do with
the behavior of the convolution semigroup as time goes to zero, while
the other has to do with what happens as time goes to infinity. This
is interesting from the functional analytic perspective in addition
to the locally compact (quantum) group perspective.
\begin{thm}
\label{thm:strong_amen_dual_non_T_Haagerup_chars}Let $\G$ be a second
countable locally compact quantum group. Then the following two conditions
are equivalent:
\begin{enumerate}
\item \label{enu:strong_amen_dual_non_T_chars__1}$\G$ is strongly amenable
(that is, $\hat{\G}$ is co-amenable) and $\hat{\G}$ does \emph{not}
have property (T);
\item \label{enu:strong_amen_dual_non_T_chars__2}there exists a symmetric,
$w^{*}$-continuous, but \emph{not norm-continuous}, convolution semigroup
of states $\left(\mu_{t}\right)_{t\ge0}$ on $\G$ such that
\begin{equation}
\big\Vert\tilde{\lambda}(\omega\star\mu_{t}-\mu_{t})\big\Vert\xrightarrow[t\to\infty]{}0\qquad(\forall\text{ state }\omega\in\Lone(\G))\label{eq:strong_amen_dual_non_T_chars__2}
\end{equation}
and $\big\Vert\tilde{\lambda}(\mu_{t})\big\Vert=1$ for each $t\ge0$.
\end{enumerate}
Furthermore, the following two conditions are equivalent:
\begin{enumerate}[resume]
\item \label{enu:strong_amen_dual_Haagerup_chars__1}$\G$ is strongly amenable
(that is, $\hat{\G}$ is co-amenable) and $\hat{\G}$ has the Haagerup
property;
\item \label{enu:strong_amen_dual_Haagerup_chars__2}there exists a symmetric,
$w^{*}$-continuous convolution semigroup of states $\left(\mu_{t}\right)_{t\ge0}$
on $\G$ such that
\[
\big\Vert\tilde{\lambda}(\omega\star\mu_{t}-\mu_{t})\big\Vert\xrightarrow[t\to\infty]{}0\qquad(\forall\text{ state }\omega\in\Lone(\G))
\]
and $\tilde{\lambda}(\mu_{t})\in\Cz(\hat{\G})$ with $\big\Vert\tilde{\lambda}(\mu_{t})\big\Vert=1$
for each $t>0$.
\end{enumerate}
\end{thm}

The implication \ref{enu:strong_amen_dual_non_T_chars__2}$\implies$\ref{enu:strong_amen_dual_non_T_chars__1}
follows directly from \cite[Theorem~6.1]{Daws_Skalski_Viselter__prop_T}
and Theorem~\ref{thm:strong_amen_chars}. The implication \ref{enu:strong_amen_dual_Haagerup_chars__2}$\implies$\ref{enu:strong_amen_dual_Haagerup_chars__1}
follows from \cite[Lemma~2.17 implication (a)$\implies$(c') and Proposition~4.9 implication (c)$\implies$(a)]{Skalski_Viselter__convolution_semigroups}
and from Theorem~\ref{thm:strong_amen_chars}. For the converse implications
we need the next spectral-theoretic lemma. Recall that the exponential
function is \emph{not} operator monotone.
\begin{lem}
\label{lem:C_e_to_minus_tA_goes_to_zero}Let $\H$ be a Hilbert space
and let $A$ be a (generally unbounded) positive selfadjoint operator
on $\H$. Consider the spectral measure $E_{A}$ of $A$ and the contractive
selfadjoint $C_{0}$-semigroup $\left(e^{-tA}\right)_{t\ge0}$ associated
to $A$. Let $C\in B(\H)$.
\begin{enumerate}
\item \label{enu:C_e_to_minus_tA_goes_to_zero__1}We have $\lim_{t\to\infty}\left\Vert Ce^{-tA}\right\Vert =0$
$\iff$ $\lim_{r\to0^{+}}\left\Vert CE_{A}\left(\left[0,r\right]\right)\right\Vert =0$.
\item \label{enu:C_e_to_minus_tA_goes_to_zero__2}Let $\mathbf{A}$ be another
(generally unbounded) positive selfadjoint operator on $\H$ such
that $A\le\mathbf{A}$ in the usual sense of forms \cite[Definition~10.5]{Schmudgen__unbounded_sa_oper_book}
and write $E_{\mathbf{A}}$ for the spectral measure of $\mathbf{A}$.
Then:
\begin{enumerate}
\item \label{enu:C_e_to_minus_tA_goes_to_zero__2_a}For every $r_{1},r_{2}>0$
we have $\left\Vert E_{A}\left([r_{2},\infty)\right)E_{\mathbf{A}}\left(\left[0,r_{1}\right]\right)\right\Vert \le\sqrt{\frac{r_{1}}{r_{2}}}$.
\item \label{enu:C_e_to_minus_tA_goes_to_zero__2_b}If the equivalent conditions
of part \ref{enu:C_e_to_minus_tA_goes_to_zero__1} hold, then also
$\lim_{t\to\infty}\left\Vert Ce^{-t\mathbf{A}}\right\Vert =0$.
\end{enumerate}
\end{enumerate}
\end{lem}

\begin{proof}
\ref{enu:C_e_to_minus_tA_goes_to_zero__1}. For $(\impliedby)$, assume
that $\lim_{r\to0^{+}}\left\Vert CE_{A}\left(\left[0,r\right]\right)\right\Vert =0$.
For all $t,r>0$ we have $\left\Vert E_{A}\left(\left(r,\infty\right)\right)e^{-tA}\right\Vert \le e^{-tr}$.
Let $\e>0$ be given, and choose $r_{0}>0$ such that $\left\Vert CE_{A}\left(\left[0,r_{0}\right]\right)\right\Vert \le\e$.
Then for all $t>0$, 
\[
\left\Vert Ce^{-tA}\right\Vert \le\left\Vert CE_{A}\left(\left[0,r_{0}\right]\right)e^{-tA}\right\Vert +\left\Vert CE_{A}\left(\left(r_{0},\infty\right)\right)e^{-tA}\right\Vert \le\e+\left\Vert C\right\Vert e^{-tr_{0}}\xrightarrow[t\to\infty]{}\e,
\]
proving that $\lim_{t\to\infty}\left\Vert Ce^{-tA}\right\Vert =0$.

For $(\implies)$, assume that $\lim_{t\to\infty}\left\Vert Ce^{-tA}\right\Vert =0$.
Let $\e>0$ be given. Pick $t_{0}>0$ such that $\left\Vert Ce^{-t_{0}A}\right\Vert \le\e$,
and set $r_{0}:=\frac{\ln2}{t_{0}}$. Then $\left\Vert e^{t_{0}A}E_{A}\left(\left[0,r_{0}\right]\right)\right\Vert \le2$,
thus $\left\Vert CE_{A}\left(\left[0,r_{0}\right]\right)\right\Vert =\left\Vert Ce^{-t_{0}A}e^{t_{0}A}E_{A}\left(\left[0,r_{0}\right]\right)\right\Vert \le2\e$.
Since the function on $[0,\infty)$ given by $r\mapsto\left\Vert CE_{A}\left(\left[0,r\right]\right)\right\Vert $
increases, this proves that $\lim_{r\to0^{+}}\left\Vert CE_{A}\left(\left[0,r\right]\right)\right\Vert =0$.

\ref{enu:C_e_to_minus_tA_goes_to_zero__2}\ref{enu:C_e_to_minus_tA_goes_to_zero__2_a}.
Let $x\in\H$ and write $y:=E_{\mathbf{A}}\left(\left[0,r_{1}\right]\right)x$.
Since $A\le\mathbf{A}$ and $y\in D(\mathbf{A}^{\frac{1}{2}})$, we
have $y\in D(A^{\frac{1}{2}})$ and 
\[
\Vert A^{\frac{1}{2}}y\Vert^{2}\le\Vert\mathbf{A}^{\frac{1}{2}}y\Vert^{2}\le r_{1}\left\Vert y\right\Vert ^{2}.
\]
On the other hand, 
\[
\Vert A^{\frac{1}{2}}y\Vert^{2}\ge\Vert A^{\frac{1}{2}}E_{A}\left([r_{2},\infty)\right)y\Vert^{2}\ge r_{2}\left\Vert E_{A}\left([r_{2},\infty)\right)y\right\Vert ^{2}.
\]
Combining these we get 
\[
\left\Vert E_{A}\left([r_{2},\infty)\right)E_{\mathbf{A}}\left(\left[0,r_{1}\right]\right)x\right\Vert \le\sqrt{\frac{r_{1}}{r_{2}}}\left\Vert E_{\mathbf{A}}\left(\left[0,r_{1}\right]\right)x\right\Vert \le\sqrt{\frac{r_{1}}{r_{2}}}\left\Vert x\right\Vert ,
\]
proving the assertion.

\ref{enu:C_e_to_minus_tA_goes_to_zero__2}\ref{enu:C_e_to_minus_tA_goes_to_zero__2_b}.
Since $CE_{\mathbf{A}}\left(\left[0,r\right]\right)=C\left(E_{A}\left(\left[0,\sqrt{r}\right]\right)+E_{A}\left(\left(\sqrt{r},\infty\right)\right)\right)E_{\mathbf{A}}\left(\left[0,r\right]\right)$,
we may apply part \ref{enu:C_e_to_minus_tA_goes_to_zero__2}\ref{enu:C_e_to_minus_tA_goes_to_zero__2_a}
with $r_{1}:=r$ and $r_{2}:=\sqrt{r}$ to infer that
\[
\left\Vert CE_{\mathbf{A}}\left(\left[0,r\right]\right)\right\Vert \le\left\Vert CE_{A}\left(\left[0,\sqrt{r}\right]\right)\right\Vert +\sqrt[4]{r}\left\Vert C\right\Vert \xrightarrow[r\to0^{+}]{}0
\]
by part \ref{enu:C_e_to_minus_tA_goes_to_zero__1}. Utilizing part
\ref{enu:C_e_to_minus_tA_goes_to_zero__1} again, this time applied
to $\mathbf{A}$ instead of $A$, we infer that $\lim_{t\to\infty}\left\Vert Ce^{-t\mathbf{A}}\right\Vert =0$.
\end{proof}

\begin{proof}[Proof of Theorem~\ref{thm:strong_amen_dual_non_T_Haagerup_chars},
continued]
\ref{enu:strong_amen_dual_non_T_chars__1}$\implies$\ref{enu:strong_amen_dual_non_T_chars__2}.
Since $\hat{\G}$ does not have property (T), there exists a symmetric,
$w^{*}$-continuous, but not norm-continuous, convolution semigroup
of states $\left(\mu^{1}_{t}\right)_{t\ge0}$ on $\G$ by \cite[Theorem~4.6]{Skalski_Viselter__convolution_semigroups}.
Since $\hat{\G}$ is co-amenable, there exists a symmetric, norm-continuous,
convolution semigroup of states $\left(\mu^{2}_{t}\right)_{t\ge0}$
on $\G$ such that $\big\Vert\tilde{\lambda}(\omega\star\mu^{2}_{t}-\mu^{2}_{t})\big\Vert\xrightarrow[t\to\infty]{}0$
for each state $\omega\in\Lone(\G)$ by Theorem~\ref{thm:strong_amen_chars}.
For $i\in\left\{ 1,2\right\} $, consider the $C_{0}$-semigroup ${(R^{\mathrm{u}}_{\mu^{i}_{t}})}_{t\ge0}$
of completely positive maps of norm $1$ on $\CzU(\G)$, and let $A_{i}$
be the positive selfadjoint operator on $\Ltwo(\G)$ such that $\tilde{\lambda}(\mu^{i}_{t})=e^{-tA_{i}}$
for each $t\ge0$ (see Theorem~\ref{thm:SV_JMPA_3.2_3.3_3.4}). Note
that since $\left(\mu^{2}_{t}\right)_{t\ge0}$ is norm continuous
but $\left(\mu^{1}_{t}\right)_{t\ge0}$ is not, we get that ${(R^{\mathrm{u}}_{\mu^{2}_{t}})}_{t\ge0}$
is norm continuous but ${(R^{\mathrm{u}}_{\mu^{1}_{t}})}_{t\ge0}$
is not, and that $A_{2}$ is bounded but $A_{1}$ is not.

We now construct $\left(\mu_{t}\right)_{t\ge0}$ and make a few observations
about it by applying the Trotter product formula (see Theorem~\ref{thm:C0_bdd_perturb})
thrice. First, apply the product formula to the $C_{0}$-semigroups
${(R^{\mathrm{u}}_{\mu^{i}_{t}})}_{t\ge0}$, $i\in\left\{ 1,2\right\} $,
on the Banach space $\CzU(\G)$, which is possible because ${(R^{\mathrm{u}}_{\mu^{2}_{t}})}_{t\ge0}$
is norm continuous, to get that for every $t\ge0$ and $a\in\CzU(\G)$,
the limit 
\[
T^{\mathrm{u}}_{t}(a):=\lim_{n\to\infty}\left(R^{\mathrm{u}}_{\mu^{1}_{\frac{t}{n}}}\circ R^{\mathrm{u}}_{\mu^{2}_{\frac{t}{n}}}\right)^{n}(a)=\lim_{n\to\infty}R^{\mathrm{u}}_{(\mu^{2}_{\frac{t}{n}}\star\mu^{1}_{\frac{t}{n}})^{\star n}}(a)=\lim_{n\to\infty}R^{\mathrm{u}}_{(\mu^{1}_{\frac{t}{n}}\star\mu^{2}_{\frac{t}{n}})^{\star n}}(a)\in\CzU(\G)
\]
exists, and $\left(T^{\mathrm{u}}_{t}\right)_{t\ge0}$ is a contractive
$C_{0}$-semigroup on $\CzU(\G)$. Since the maps $R^{\mathrm{u}}_{\mu^{1}_{t}}$
and $R^{\mathrm{u}}_{\mu^{2}_{t}}$, $t\ge0$, are completely positive
and commute with $L^{\mathrm{u}}_{\nu}$, $\nu\in\CzU(\G)^{*}$, the
maps $T^{\mathrm{u}}_{t}$, $t\ge0$, have these properties as well.
In addition, since ${(R^{\mathrm{u}}_{\mu^{2}_{t}})}_{t\ge0}$ is
norm continuous but ${(R^{\mathrm{u}}_{\mu^{1}_{t}})}_{t\ge0}$ is
not, $\left(T^{\mathrm{u}}_{t}\right)_{t\ge0}$ is not norm continuous
either. Consequently, Theorem~\ref{thm:SV_JMPA_3.2_3.3_3.4} yields
a $w^{*}$-continuous, but not norm-continuous, convolution semigroup
$\left(\mu_{t}\right)_{t\ge0}$ of \emph{contractive positive functionals}
on $\G$ such that $T^{\mathrm{u}}_{t}:=R^{\mathrm{u}}_{\mu_{t}}$
for each $t\ge0$. From \cite[Lemmas~2.3 and 2.17]{Skalski_Viselter__convolution_semigroups}
we infer that for every $t\ge0$, 
\begin{equation}
\mu_{t}=w^{*}\text{-}\lim_{n\to\infty}\big(\mu^{2}_{\frac{t}{n}}\star\mu^{1}_{\frac{t}{n}}\big)^{\star n}=w^{*}\text{-}\lim_{n\to\infty}\big(\mu^{1}_{\frac{t}{n}}\star\mu^{2}_{\frac{t}{n}}\big)^{\star n}\label{eq:Trotter1_mu_t}
\end{equation}
and 
\begin{equation}
\tilde{\lambda}(\mu_{t})=\mathrm{WOT}\text{-}\lim_{n\to\infty}\tilde{\lambda}\big((\mu^{1}_{\frac{t}{n}}\star\mu^{2}_{\frac{t}{n}})^{\star n}\big)=\mathrm{WOT}\text{-}\lim_{n\to\infty}\big(\tilde{\lambda}(\mu^{1}_{\frac{t}{n}})\tilde{\lambda}(\mu^{2}_{\frac{t}{n}})\big)^{n}.\label{eq:Trotter1_lambda_tilde_mu_t}
\end{equation}
From (\ref{eq:Trotter1_mu_t}) it follows that $\left(\mu_{t}\right)_{t\ge0}$
is symmetric. On the other hand, applying the Trotter product formula
to the contractive $C_{0}$-semigroups $(\tilde{\lambda}(\mu^{i}_{t})=e^{-tA_{i}})_{t\ge0}$,
$i\in\left\{ 1,2\right\} $, on $\Ltwo(\G)$, which is possible because
$A_{2}$ is bounded, we obtain for every $t\ge0$,
\begin{equation}
e^{-t(A_{1}+A_{2})}=\mathrm{SOT}\text{-}\lim_{n\to\infty}\big(e^{-\frac{t}{n}A_{1}}e^{-\frac{t}{n}A_{2}}\big)^{n}=\mathrm{SOT}\text{-}\lim_{n\to\infty}\big(\tilde{\lambda}(\mu^{1}_{\frac{t}{n}})\tilde{\lambda}(\mu^{2}_{\frac{t}{n}})\big)^{n}\qquad(\forall t\ge0).\label{eq:Trotter2}
\end{equation}
We deduce from (\ref{eq:Trotter1_lambda_tilde_mu_t}) and (\ref{eq:Trotter2})
that $\tilde{\lambda}(\mu_{t})=e^{-t(A_{1}+A_{2})}$ for all $t\ge0$. 

We now show that $\left(\mu_{t}\right)_{t\ge0}$ consists of states,
which is equivalent to the operator $\tilde{\lambda}(\mu_{t})$ having
norm 1 for each $t\ge0$ since $\hat{\G}$ is co-amenable, as in the
proof of Theorem~\ref{thm:strong_amen_chars} implication \ref{enu:strong_amen_chars__1}$\implies$\ref{enu:strong_amen_chars__4}.
By \cite[Lemma~2.17]{Skalski_Viselter__convolution_semigroups}, for
$i\in\left\{ 1,2\right\} $, the semigroup $(\tilde{\lambda}(\mu^{i}_{t}))_{t\ge0}$
is, in fact, continuous in the strict topology of $\M(\Cz(\hat{\G}))$,
and thus it can be viewed as a (contractive) $C_{0}$-semigroup on
the Banach space $\Cz(\hat{\G})$. Applying the Trotter product formula
to these $C_{0}$-semigroups on $\Cz(\hat{\G})$ yields that for every
$t\ge0$ the limit 
\[
\lim_{n\to\infty}\big(\tilde{\lambda}(\mu^{1}_{\frac{t}{n}})\tilde{\lambda}(\mu^{2}_{\frac{t}{n}})\big)^{n}
\]
actually exists (and necessarily equals $\tilde{\lambda}(\mu_{t})$
by (\ref{eq:Trotter1_lambda_tilde_mu_t})) in the strict topology
of $\M(\Cz(\hat{\G}))$. Using that $\hat{\G}$ is co-amenable, we
may therefore apply the co-unit $\hat{\epsilon}\in\Cz(\hat{\G})^{*}$,
or to be precise, its strict extension to $\M(\Cz(\hat{\G}))$, to
this limit to obtain that for each $t\ge0$,
\[
\mu_{t}(\one)=\hat{\epsilon}(\tilde{\lambda}(\mu_{t}))=\lim_{n\to\infty}\hat{\epsilon}\left[\big(\tilde{\lambda}(\mu^{1}_{\frac{t}{n}})\tilde{\lambda}(\mu^{2}_{\frac{t}{n}})\big)^{n}\right]=\lim_{n\to\infty}\left(\hat{\epsilon}\big(\tilde{\lambda}(\mu^{1}_{\frac{t}{n}})\big)\hat{\epsilon}\big(\tilde{\lambda}(\mu^{2}_{\frac{t}{n}})\big)\right)^{n}=1
\]
for $\left(\mu^{1}_{t}\right)_{t\ge0}$ and $\left(\mu^{2}_{t}\right)_{t\ge0}$
consist of states. This proves that $\left(\mu_{t}\right)_{t\ge0}$
also consists of states, as claimed.

Finally, we prove (\ref{eq:strong_amen_dual_non_T_chars__2}). If
$\omega\in\Lone(\G)$ is a state, then since $\big\Vert\big(\tilde{\lambda}(\omega)-\one\big)e^{-tA_{2}}\big\Vert=\big\Vert\tilde{\lambda}(\omega\star\mu^{2}_{t}-\mu^{2}_{t})\big\Vert\xrightarrow[t\to\infty]{}0$,
Lemma~\ref{lem:C_e_to_minus_tA_goes_to_zero} \ref{enu:C_e_to_minus_tA_goes_to_zero__2}\ref{enu:C_e_to_minus_tA_goes_to_zero__2_b}
applies with $A:=A_{2}$, $\mathbf{A}:=A_{1}+A_{2}$ (as $A_{2}$
is bounded), and $C:=\tilde{\lambda}(\omega)-\one$ to yield that
also $\big\Vert\tilde{\lambda}(\omega\star\mu_{t}-\mu_{t})\big\Vert=\big\Vert\big(\tilde{\lambda}(\omega)-\one\big)e^{-t(A_{1}+A_{2})}\big\Vert\xrightarrow[t\to\infty]{}0$.
That is, (\ref{eq:strong_amen_dual_non_T_chars__2}) holds true.

\ref{enu:strong_amen_dual_Haagerup_chars__1}$\implies$\ref{enu:strong_amen_dual_Haagerup_chars__2}.
Similarly to the above proof of the implication \ref{enu:strong_amen_dual_non_T_chars__1}$\implies$\ref{enu:strong_amen_dual_non_T_chars__2}
we construct two symmetric, $w^{*}$-continuous convolution semigroups
of states on $\G$, $\left(\mu^{1}_{t}\right)_{t\ge0}$ and $\left(\mu^{2}_{t}\right)_{t\ge0}$.
The latter is constructed exactly as above using Theorem~\ref{thm:strong_amen_chars}.
Since $\hat{\G}$ has the Haagerup property, there exists a symmetric,
$w^{*}$-continuous convolution semigroup of states $\left(\mu^{1}_{t}\right)_{t\ge0}$
on $\G$ such that $\tilde{\lambda}(\mu^{1}_{t})\in\Cz(\hat{\G})$
for each $t>0$ by \cite[Theorem~4.12]{Skalski_Viselter__convolution_semigroups}.
We may and do assume that $\hat{\G}$ is not compact, hence $\left(\mu^{1}_{t}\right)_{t\ge0}$
is not norm continuous; in particular, everything established in the
proof of \ref{enu:strong_amen_dual_non_T_chars__1}$\implies$\ref{enu:strong_amen_dual_non_T_chars__2}
applies to this part of the proof too. As above, for $i\in\left\{ 1,2\right\} $,
let $A_{i}$ be the positive selfadjoint operator on $\Ltwo(\G)$
with $\tilde{\lambda}(\mu^{i}_{t})=e^{-tA_{i}}$ for all $t\ge0$,
and let $\left(\mu_{t}\right)_{t\ge0}$ be the symmetric, $w^{*}$-continuous
convolution semigroup of states on $\G$ such that $\tilde{\lambda}(\mu_{t})=e^{-t(A_{1}+A_{2})}$
for all $t\ge0$.

It is left to prove that $\tilde{\lambda}(\mu_{t})\in\Cz(\hat{\G})$
for all $t>0$. To this end we use the Dyson--Phillips theorem (see
Theorem~\ref{thm:C0_bdd_perturb}). Recursively define a family $\left(S_{n}(t)\right)_{t\ge0,n\in\Z_{+}}$
of bounded operators on $\Ltwo(\G)$ by $S_{0}(t):=e^{-tA_{1}}$ and
\[
S_{n+1}(t):=\mathrm{SOT}\text{-}\int^{t}_{0}S_{n}(t-s)A_{2}e^{-sA_{1}}\d s.
\]
Then for every $t\ge0$ we have $e^{-t(A_{1}+A_{2})}=\sum^{\infty}_{n=0}S_{n}(t)$,
where the series converges in the (operator) norm of $B(\Ltwo(\G))$.
Fix $t>0$. To prove that the operator $e^{-t(A_{1}+A_{2})}\in\M(\Cz(\hat{\G}))$
actually belongs to $\Cz(\hat{\G})$, we show that for some approximate
identity $\left(e_{j}\right)_{j\in\mathcal{J}}$ of $\Cz(\hat{\G})$
we have $e^{-t(A_{1}+A_{2})}(\one-e_{j})\xrightarrow[j\in\mathcal{J}]{}0$
in the norm topology of $\Cz(\hat{\G})$. It suffices to show that
$\left\Vert S_{n}(t)(\one-e_{j})\right\Vert \xrightarrow[j\in\mathcal{J}]{}0$
for every $n\in\Z_{+}$. This is already known for $n=0$. Fix $n\in\Z_{+}$.
By SOT-continuity of $S_{n}(\cdot)$ we have $M:=\sup_{0\le s\le t}\left\Vert S_{n}(s)\right\Vert <\infty$.
Let $\e>0$ be given and set $t_{0}:=\min(\e,t)$. Since $\left\Vert e^{-t_{0}A_{1}}(\one-e_{j})\right\Vert \xrightarrow[j\in\mathcal{J}]{}0$,
there exists $j_{0}\in\mathcal{J}$ such that for all $j\ge j_{0}$
we have $\left\Vert e^{-t_{0}A_{1}}(\one-e_{j})\right\Vert \le\e$.
Consequently, for all $s\ge t_{0}$ and $j\ge j_{0}$ we have $\left\Vert e^{-sA_{1}}(\one-e_{j})\right\Vert \le\e$.
As a result, for all $j\ge j_{0}$, 
\begin{multline*}
\left\Vert S_{n+1}(t)(\one-e_{j})\right\Vert \\
\begin{split} & =\left\Vert \mathrm{SOT}\text{-}\int^{t_{0}}_{0}S_{n}(t-s)A_{2}e^{-sA_{1}}(\one-e_{j})\d s+\mathrm{SOT}\text{-}\int^{t}_{t_{0}}S_{n}(t-s)A_{2}e^{-sA_{1}}(\one-e_{j})\d s\right\Vert \\
 & \le M\left\Vert A_{2}\right\Vert \e(1+t-t_{0}).
\end{split}
\end{multline*}
This proves the desired conclusion.
\end{proof}

\begin{rem}
In the above proof of Theorem~\ref{thm:strong_amen_dual_non_T_Haagerup_chars}
we could have avoided the usage of the convolution operators $L^{\mathrm{u}}_{\mu},R^{\mathrm{u}}_{\mu}$
on $\CzU(\G)$ at the cost of using the Dirichlet form machinery of
\cite[Theorem~3.4]{Skalski_Viselter__convolution_semigroups} (see
Remark succeeding Theorem~\ref{thm:SV_JMPA_3.2_3.3_3.4}). This would
have made the proof a little shorter but more technically demanding.
\end{rem}

\section{Open questions}

In the sequel $\G$ denotes a second countable locally compact quantum
group. The first open question just fine tunes Question~\ref{ques:amen}
in the spirit of Theorem~\ref{thm:strong_amen_dual_non_T_Haagerup_chars}.
\begin{question}
Is it true that $\G$ is amenable and $\hat{\G}$ does not have property
(T) $\iff$ there exists a $w^{*}$-continuous, but not norm-continuous,
convolution semigroup of states $\left(\mu_{t}\right)_{t\ge0}$ on
$\G$ such that $\left\Vert \omega\star\mu_{t}-\mu_{t}\right\Vert \xrightarrow[t\to\infty]{}0$
for every state $\omega\in\Lone(\G)$?

Is it true that $\G$ is amenable and $\hat{\G}$ has the Haagerup
property $\iff$ there exists a $w^{*}$-continuous convolution semigroup
of states $\left(\mu_{t}\right)_{t\ge0}$ on $\G$ such that $\left\Vert \omega\star\mu_{t}-\mu_{t}\right\Vert \xrightarrow[t\to\infty]{}0$
for every state $\omega\in\Lone(\G)$ and such that $\lambda(\mu_{t})$
(or $\tilde{\lambda}(\mu_{t})$) belongs to $\Cz(\hat{\G})$ for each
$t>0$?
\end{question}

The second question has to do with a characterization of co-amenability
of $\hat{\G}$ in terms of convolution semigroups that may appear
more natural than that of Theorem~\ref{thm:strong_amen_chars}.
\begin{question}
\label{ques:strong_amen_in_terms_of_hat_G}\mbox{}
\begin{enumerate}
\item Is it true that $\G$ is strongly amenable (that is, $\hat{\G}$ is
co-amenable) $\iff$ there exists a $w^{*}$-continuous convolution
semigroup of states $\left(\hat{\mu}_{t}\right)_{t\ge0}$ on $\hat{\G}$
such that $\hat{\mu}_{t}\in\Lone(\hat{\G})$ for all $t>0$? 
\item Is it true that both $\G,\hat{\G}$ are strongly amenable $\iff$
there exists a $w^{*}$-continuous convolution semigroup of states
$\left(\mu_{t}\right)_{t\ge0}$ on $\G$ such that $\big\Vert\tilde{\lambda}(\omega\star\mu_{t}-\mu_{t})\big\Vert\xrightarrow[t\to\infty]{}0$
for every state $\omega\in\Lone(\G)$ and such that $\tilde{\lambda}(\mu_{t})$
has norm 1 and $\mu_{t}\in\Lone(\G)$ for all $t>0$? 
\end{enumerate}
\end{question}

Recalling that locally compact groups are always co-amenable, we ask
in particular whether the first (respectively, second) part of Question~\ref{ques:strong_amen_in_terms_of_hat_G}
holds true if $\G$ is the dual of a locally compact group (respectively,
either an amenable locally compact group or its dual).

Last but not least, we express our interest in finding a characterization
in a similar flavor for \emph{weak amenability}. That should be considerably
more difficult because this approximation property is not defined
in terms of states.

\section*{Acknowledgments}

The author thanks Adam Skalski for providing valuable comments on
a draft of the manuscript.

\section*{Use of AI tools}

The author acknowledges the use of Gemini 3.5/3.6 Flash (Google) and
the free version of ChatGPT based on GPT-5.5 (OpenAI) during the preparation
of this manuscript. These tools were used to assist with some technical
mathematical points and to improve the clarity, language, and presentation
of the manuscript. The author independently verified all mathematical
arguments and results and takes full responsibility for the accuracy
and content of this paper.

\bibliographystyle{amsplain}
\bibliography{AmenConvSemig}

\end{document}